\documentclass[12pt]{amsart}
\newtheorem{theorem}{Theorem}
\newtheorem*{theorem*}{Theorem}

\newtheorem*{problem1}{Problem 66}
\newtheorem*{problem2}{Problem 67}

\newtheorem{proposition}{Proposition}
\usepackage{xcolor}
\theoremstyle{definition}
\theoremstyle{remark}

\title[Problems 66 and 67 of S\'{a}rk\"{o}zy]{Problems 66 and 67 on sums of residue classes and primes
  of Andr\'{a}s S\'{a}rk\"{o}zy's
  collection of unsolved problems}
\author{Yuchen Ding, Christian Elsholtz, Yu-Chen Sun}
\date{\today}

\address{(Yuchen Ding) School of Mathematics,  Yangzhou University, Yangzhou 225002, People's Republic of China}
\email{ycding@yzu.edu.cn}

\address {(Christian Elsholtz) Institute of Analysis and Number Theory, Graz University of Technology, Kopernikusgasse 24/II,
8010 Graz, Austria}
\email{\tt elsholtz@math.tugraz.at}

\address {(Yu--Chen Sun) School of Mathematics, University of Bristol, Bristol, BS8 1UG, England}
\email{\tt yuchensun93@163.com}

	
    
\begin{document}

\begin{abstract}
		In this note, we discuss two problems of S\'{a}rk\"{o}zy (2001). 
		In particular, we prove an optimal result on sumsets of sparse subset of primes.
	\end{abstract}

\maketitle

\section{S\'{a}rk\"{o}zy's problems}
In this paper we look at two problems from 
the collection of open problems by Andr\'{a}s S\'{a}rk\"ozy (2001) \cite{Sarkozy}.

\begin{problem1}
 Is it true, that if $\varepsilon > 0, m \in  \mathbb{N}$, 
$m$ is even, $m \rightarrow \infty$, and $A$ is
a set of distinct modulo $m$ reduced\footnote{The 
word ``reduced"
is missing here, but it is required, for details see the beginning of section \ref{section:problem66}. Also, 
S\'{a}rk\"{o}zy compares with $\varphi(m)$, which is the number of reduced residue classes modulo $m$.} residue classes  with
$
|A| >\Big(\frac{1}{2}+ \varepsilon\Big)\varphi(m),
$
then\\
a) $A + A$ contains almost all even residue classes modulo $m$;\\
b) $A + A + A$ contains every odd residue class modulo $m$?
\end{problem1}

\begin{problem2} 
Is it true that if $Q = \{q_1, q_2, \ldots \}$ is an infinite set of primes such that
$\liminf_{x\rightarrow \infty}
\frac{\log x}{x} |\{q_i : q_i \leq x, q_i \in  Q\}| > \frac{1}{2},$
then every large odd integer $2n + 1$ can be represented in the form
$
q_i + q_j + q_k=2n+1 ~(\text{with~} q_i,q_j,q_k \in Q) ?
$
\end{problem2}

Here $A+A=\{a_i+a_j: a_i, a_j\in A\}$
is the usual sumset notation, similarly $Q-W=\{q_i-w_j:q_i \in Q, w_j \in W\}$.

Problem 66 can be understood as an analogue of the well known problems of sums of two primes (in part a) and  sums of three primes (in part b) in the positive integers, to the group $\mathbb{Z}_m^*$ of reduced residue classes modulo $m$. For sums of two primes the famous Goldbach conjecture states that every even integer $n \geq 4$ can be written as a sum of two primes. This conjecture is still open. There are many partial results proving it for almost all even numbers, including a result by Pintz \cite{Pintz} saying that the number of possible
exceptions below $X$ is at most $X^{3/4}$, for sufficiently large $X$.

The so called ternary Goldbach conjecture states, that every odd integer $n \geq 7$ can be written as a sum of three primes. This was proven, under the additional assumption of the generalized Riemann Hypothesis (GRH)  for all but at most finitely many exceptions by Hardy and Littlewood (1923), and unconditionally for all but at most finitely many exceptions by I.M. Vinogradov (1937) \cite{Vinogradov}.
A complete proof of the ternary Goldbach problem was  announced by Harald Helfgott \cite{Helfgott}, the final version is yet to appear.

These very well known problems in the integers have their origin in a correspondence of 1742 between Christian Goldbach and Leonhard Euler, even though the precise statements were a bit different: at that time the integer 1 was considered to be a prime, whereas today one does not consider 1 as a prime, by definition, as it would make statements about  unique factorization much more clumsy.

An analogue of the binary Goldbach problem has been proved in the ring of polynomials, and has been a popular theme \cite{Hayes, Kozek, Pollack, Saidak} and see also \cite{Effinger-Hayes}. Note that in these cases a positive proportion of the polynomials is irreducible. 

For  Problem 67 subsets of the primes now come into the play.
Let $\pi(x)$ denote the number of primes $p\leq x$.
The function $\pi(x)$ grows
asymptotically like $\frac{x}{\log x}$.
S\'{a}rk\"{o}zy's question 67 is a very early example of a question on sums of primes with a positive density, (relative to the set of all primes).
(In fact, we are aware of some growing literature on this topic, starting in 2011, see e.g.~\cite{Alsetri-1, Alsetri, Chipeniuk-Hamel,  Cui-Li-Xue, Gao, Matomaki, Shao, Yang-Togbe}.)
Problem 67 can be understood as asking whether the Vinogradov approach (using the circle) method can be extended if one only uses a bit more than half of all primes, or whether other serious obstacles could occur.

In \cite{Yang-Togbe} Yang and  Togb\'{e} answered negatively  Problem 67 and obtained more explicit results on this problem. We will make further considerations of it later in Section 3.

\section{Problem 66}{\label{section:problem66}}
There is a small inaccuracy in S\'{a}rk\"{o}zy's original statement.
Without the word  ``reduced" simple counter examples are products of the first primes:
$m= p_1 p_2 \cdots p_t$. Here it is well known that $\varphi(m) \sim e^{-\gamma} \frac{m}{\log \log m}$, where $e=2.718\ldots$ 
is the base of the natural logarithm and 
$\gamma=0.577\ldots $ is the Euler-Mascheroni constant. With such small values of $\varphi(m)$ the choice ${\mathcal A}=\{i \mod m: 1 \leq i \leq \varphi(m)\}$
shows that $|{\mathcal A}+ {\mathcal A}|=2\varphi(m)-1$ and $|{\mathcal A}+ {\mathcal A}+{\mathcal A}|=3\varphi(m)-2$. As $m=p_1 \cdots p_t\rightarrow \infty$, the proportion of residue classes represented by $A+A$ or $A+A+A$ respectively, tends to $0$.

Next let us first look at the ternary case b) for which we not only find counter examples to the question, but also find an essentially  best possible answer to the problem. 
As the problem is about a ternary sum of reducible residue classes it is natural to compare this with results on the sum of three primes.

In his work on a density result to the ternary Goldbach problem Shao \cite{Shao} gave the following example.
Let $A_1=\{1,2,4,7,13 \}\subset \mathbb{Z}/(15 \mathbb{Z})$.\footnote{
We remark that
in a later paper Shen \cite{Shen} used
$A_2=\{1,4,7,11,13\}$ to avoid the class $2  \mod 15$ in $A_2+A_2+A_2$. This second example is isomorphic to the first one, as a multiplication by $13$ shows.}
Then $|A_1|> \frac{\varphi(15)}{2}=4$, but the residue class $14\pmod {15}$
is not in $A_1+A_1+A_1$, which would answer negatively Problem 66b via slight adjustments according to its requirements.
As S\'{a}rk\"{o}zy's question is about large even $m$, we modify the example above as follows.
Let $p$ be a sufficiently large prime, let $m=30p$ and therefore $\varphi(m)=8(p-1)$, and let
$$A=\{30k+1,30k+7,30k+13,30k+17,30k+19: 1\le 30k\le m\}.$$
For each  $i\in \{1,7,13,17,19\}$ there is at most one $k$ with the constraint $1\le 30k\le m$ satisfying
$30k+i\equiv 0\pmod{p}$, leading to the bound
$$ |A|\geq \frac{m}{6}-5= \frac{5}{8}\varphi(m).$$
One easily checks that
$29\pmod{30}$ cannot be written as the sum of $A_0+A_0+A_0$, where
$$A_0=\{1,7,13,17,19\pmod{30}\}.$$
As a simple consequence, we know that no member of $A+A+A$ is contained in 
$\{30k+29\pmod{m}:1\le 30k\le m\}$,
answering negatively problem 66b).

Next we come to Problem 66a).
Here we have an even simpler counter example, based on residue classes modulo 12. Let $p$ be a sufficiently large prime, $m=12p$, and
$$A=\{12k+1,12k+5,12k+7: 1\le 12k\le m\}.$$
Similar discussions as above lead to $|A|\ge 2\varphi(m)/3$ as well as that no element of 
$$
\{12k+4\pmod{m}:1\le 12k\le m\}
$$
is in $A+A$, which yields a proportion of 1/6 of all the even residue classes modulo $m$ not covered by $A+A$.  

The following quite different type of counter example to Problem 66a) may be of independent interest.
If a class $a$ in $\{1,2,4,7,13\}$ is even we replace it by $a+15$.
This gives the set $A=\{1,7,13,17,19 \}$, understood as residue classes modulo $30$.
Then $|A|> \frac{\varphi(30)}{2}=4$, but the residue classes 
$$
B=\{10,12,16,22,28\pmod{30} \}
$$
are not in $A+A$. We obtain the infinite family $m_k= 30 \cdot 2^k$.
Choose 
$$
A_1=A \cup (A+30)\subset\mathbb{Z}/(60 \mathbb{Z})
$$
and for any $k\ge 1$,
\[\quad A_{k+1}= A_k \cup (A_k+2^k\cdot 30)
\subset \mathbb{Z}/(30\cdot 2^{k+1} \mathbb{Z}).\]
Hence $|A_{k}|= 5 \cdot 2^k >\frac{1}{2}\varphi(30 \cdot 2^k)=2^{k+2}$, and $A_k+A_k$ misses out a positive proportion of the even residue classes modulo $m_k$.
This provides another counter example of Problem  66a).

We point out that S\'{a}rk\"ozy's predicition, i.e., Problem 66 b), is not far from the truth. Actually, Shao \cite[Corollary 1.5]{Shao} proved the following very interesting result.
\begin{proposition}\label{shao-theorem}
Let $m$ be a square-free positive odd integer. Let $A$ be a
subsets of $\mathbb{Z}_m^* $ with $|A| > \frac{5}{8}\varphi(m).$
Then
$A + A + A= \mathbb{Z}_m.$ 
\end{proposition}

Here $\mathbb{Z}_m=\mathbb{Z}/m\mathbb{Z}$ and $\mathbb{Z}_m^* =(\mathbb{Z}/m\mathbb{Z})^*$.
This gives a best possible result for odd $m$,
which is sharp, for example when $m=15$. 

As a conclusion, we now have the following theorem.
\begin{theorem}\label{thm1}
The answers to Problem 66 a) and b) are both negative. Moreover, if we replace $1/2$ and even $m$ by $5/8$ and odd squre-free $m$ respectively, the answer to b) will be positive. 
\end{theorem}

It is here also worth mentioning a more general version proved by 
Shen \cite[Corollary 1.4]{Shen}.

{\it Let $m$ be a square-free positive odd integer. Let $A_1, A_2, A_3$ be
three subsets of $\mathbb{Z}_m^* $ with $|A_1| > \frac{5}{8}\varphi(m),
|A_i| \geq  \frac{5}{8}\varphi(m)\,  (i = 2, 3)$.
Then
$$A_1 + A_2 + A_3 = \mathbb{Z}_m.$$ }

\section{Problem 67}
This question is on primes, but as almost all of the primes lie in the reduced residue classes modulo any fixed integer $m$ the two problems are closely connected.
As pointed out by Yang and Togb\'{e} \cite{Yang-Togbe}, the negative  answer of Problem 67 can be concluded from Shao \cite{Shao} as follows. Let $Q$ be the set of primes in the residue classes
$1,2,4,7,13 \pmod {15}$. Then the counting function of elements $p \leq x$ of $Q$
is $Q(x)\sim \frac{5}{8}\frac{x}{\log x}> \frac{1}{2}\frac{x}{\log x}$, by Dirichlet's theorem on primes in arithmetic progressions \cite{Davenport}.
As the residue class 14 modulo 15 
cannot be written as a ternary sum of $\{1,2,4,7,13\}$
there is even a positive proportion of all integers which are counter examples to this question. 

We now state Shao's important  result below for further discussions.
\begin{proposition}[Shao \cite{Shao}]
If $Q = \{q_1, q_2, \ldots \}$ is an infinite set of primes such that
\begin{align}\label{eq-des-1}
\liminf_{x\rightarrow \infty}
\frac{\log x}{x} |\{q_i : q_i \leq x, q_i \in  Q\}| > \frac{5}{8},
\end{align}
then every large odd integer $2n + 1$ can be represented in the form
$
q_i + q_j + q_k=2n+1 ~(\text{with~} q_i,q_j,q_k \in Q).
$
\end{proposition}

In \cite{Yang-Togbe}, Yang and Togb\'{e} also constructed a set $Q$ of primes with upper relative prime density $1$ such that there are infinitely many odd integers $n$ that cannot be represented as 
$
q_1 + q_2 + q_3=2n+1
$
with $q_i\in Q~(1\le i\le 3)$. Moreover, the set $Q$ in their construction has lower relative prime density $1/3$. It now seems of interest to step a little further on the counter example of  Problem 67, pursuing the maximal value of the lower relative prime density.  

We now state our main result, which may appear surprising, in view of the proposition above.
\begin{theorem}\label{thm2}
There is an infinite set $Q$ of primes with upper relative prime density $1$ and lower relative prime density $5/8$ such that there are infinitely many odd integers $2n+1$ that cannot be represented as 
$
q_1 + q_2 + q_3=2n+1
$
with $q_i\in Q~(1\le i\le 3)$.
\end{theorem}
\begin{proof}
Let $\mathcal{P}$ be the set of primes and $S$ any subset of $\mathcal{P}$. For any $x$, let
$S(x)$ be the number of elements of $S$ not exceeding $x$.
Define
$$
\overline{d}(S)=\limsup_{x\rightarrow \infty}
\frac{\log x}{x} |\{s_i : s_i  \leq x,~  s_i \in  S\}|
$$
and
$$
\underline{d}(S)=\liminf_{x\rightarrow \infty}
\frac{\log x}{x} |\{s_i : s_i  \leq x,~  s_i \in  S\}|.
$$
We will construct the desired set $Q$ step by step.

{\it Step 1.} We first construct a set $H\subset \mathcal{P}$ with $\overline{d}(H)=1$ and $\underline{d}(H)=\frac 58$.
Let $x_1$ be a sufficiently large number and
$
x_{2}=e^{e^{x_1}}.
$
Define
$$
H_1=(x_1,x_2]\cap\mathcal{P}.
$$
Let $d_2=e^{e^{x_2}}$.
Then we choose some odd integer $x_3>d_{2}$ satisfying 
$
x_3\equiv 29\pmod{30}.
$
Let
$$
H_2=\big\{p\in \mathcal{P}: x_2<p\le x_{3},~ p\equiv a\!\!\!\pmod{15}~\text{for~some~}a\in A_1\big\},
$$
where $A_1=\{1,2,4,7,13\}$ is the set given by Shao \cite{Shao}. 
Now, suppose that we have already gave the definitions of $H_j$ and $x_{j+1}$ for $j\le 2k$. Then we define $H_{2k+1},H_{2k+2}, x_{2k+2}$, and $x_{2k+3}$ recursively  as follows. Define
$$
x_{2k+2}=e^{e^{x_{2k+1}}} \quad
\text{and} \quad 
H_{2k+1}=(x_{2k+1},x_{2k+2}]\cap\mathcal{P}.
$$
Let 
$
d_{2k+2}=e^{e^{x_{2k+2}}}.
$
Then we choose some odd integer $x_{2k+3}>d_{2k+2}$ satisfying $x_{2k+3}\equiv 29\pmod{30}$.
Define
$$
H_{2k+2}=\big\{p\in \mathcal{P}: x_{2k+2}<p\le x_{2k+3}, ~p\equiv a\!\!\!\!\pmod{15}~\text{for~some~}a\in A_1\big\}.
$$
Let $H=\cup_{k=1}^{\infty}H_k$. Then we clearly have $\overline{d}(H)=1$ and $\underline{d}(H)=\frac 58$ from Dirichlet's theorem in arithmetic progressions.

{\it Step 2.} We next drop some elements of $H$ to form $W$ such that
$$W(x)\sim H(x), \quad \text{as~}x\rightarrow\infty.$$
As a consequence, we have $W\subset \mathcal{P}$ with $\overline{d}(W)=1$ and $\underline{d}(W)=\frac 58$.

For any positive integer $k$, let $W_{2k-1}=H_{2k-1}$ and
\begin{align}\label{W-H}
W_{2k}=H_{2k}\setminus \bigg[x_{2k+1}-x_{2k}-\frac{x_{2k+1}}{\sqrt{\log x_{2k+1}}},x_{2k+1}\bigg].
\end{align}
Let $W=\cup_{k=1}^{\infty}W_k$.
It is easy to see $W(x)\sim H(x)$ as $x\rightarrow\infty$ since
$$
\bigg|\Big[x_{2k+1}-x_{2k}-\frac{x_{2k+1}}{\sqrt{\log x_{2k+1}}},x_{2k+1}\Big]\cap\mathcal{P}\bigg|=o\Big(\frac{x_{2k+1}}{\log x_{2k+1}}\Big).
$$

{\it Step 3.} We continue dropping some elements of $W$ to form our desired set $Q$. 
For any positive integer $k$, let $Q_{2k-1}=W_{2k-1}=H_{2k-1}$.
For any prime $p$ with $x_1<p\le x_{2k}$, we consider the representations of the difference
$
x_{2k+1}-p
$ 
as the sum of two primes in $W_{2k}$.
Suppose that
\begin{align}\label{uv}
x_{2k+1}-p=w_1+v_1=w_2+v_2=\cdots=w_{t_p}+v_{t_p} 
\end{align}
with $w_i, v_i\in W_{2k}$ are all the representations of 
$
x_{2k+1}-p
$
as the sum of two elements of $W_{2k}$. We now let
\begin{align}\label{lastone}
Q_{2k}=W_{2k}\setminus \big\{w_i,v_i: 1\le i\le t_p,~x_1< p\le x_{2k} \big\}.
\end{align}
Finally, we define $Q=\cup_{k=1}^{\infty}Q_k$.

We will prove the following facts, from which we complete the proof. 

{\it Fact 1. We have $\overline{d}(Q)=1$ and $\underline{d}(Q)=\frac 58$.} 

For this fact,
it suffices to prove
$$Q(x)\sim W(x), \quad \text{as~}x\rightarrow\infty.$$
Following a standard result \cite[Theorem 7.2]{Nathanson} on the number of representations of an integer as a sum of two primes, which can be deduced from Selberg's upper bound sieve,
the following estimate holds:
\begin{align}\label{number-eq-1}
t_p\ll \frac{x_{2k+1}}{(\log x_{2k+1})^2}\prod_{p'|(x_{2k+1}-p)}\Big(1+\frac1{p'}\Big)\ll  \frac{x_{2k+1}\log\log x_{2k+1}}{(\log x_{2k+1})^2}
\end{align}
for any $x_1< p\le x_{2k}$, where the implied constants are absolute.
From (\ref{number-eq-1}), we clearly have
\begin{align}\label{nanjing-1}
|Q_{2k}-W_{2k}|&\leq \sum_{x_1<p\le x_{2k}}t_p\nonumber\\
&\ll
\frac{x_{2k}x_{2k+1}\log\log x_{2k+1}}{(\log x_{2k+1})^2}\nonumber\\
&\ll \frac{x_{2k+1}(\log\log x_{2k+1})^2}{(\log x_{2k+1})^2}.
\end{align}
For any $x_1< p\le x_{2k}$ and $1\le i\le t_p$ we see from (\ref{W-H}) and (\ref{uv}) that 
\begin{align}\label{nanjing-2}
w_i,v_i>\frac{x_{2k+1}}{\sqrt{\log x_{2k+1}}}.
\end{align}
We conclude from (\ref{nanjing-1}) and (\ref{nanjing-2}) that
$Q(x)\sim W(x),$ as $x\rightarrow\infty$
since 
$$
\pi\bigg(\frac{x_{2k+1}}{\sqrt{\log x_{2k+1}}}\bigg)\gg \frac{x_{2k+1}}{(\log x_{2k+1})^{3/2}},
$$
where $\pi(x)$ is the number of primes not exceeding $x$.

{\it Fact 2. For sufficiently large $k$, we have 
$
x_{2k+1}\not\in Q+Q+Q.
$}

Assume the contrary, we have
$
x_{2k+1}=q_1+q_2+q_3~ (q_1\le q_2\le q_3).
$
We separate our discussions into three case.

{\it Case I.} $q_1,q_2,q_3\in Q_{2k}$.
This is an apparent contradiction since any odd integer $29\pmod{30}$ is not contained in the ternary sumset $Q_{2k}+Q_{2k}+Q_{2k}$.

{\it Case II.} $q_1,q_2\not\in Q_{2k}$ and $q_3\in Q_{2k}$. By the construction, we have
$$
q_1+q_2+q_3\le 2x_{2k}+\Big(x_{2k+1}-x_{2k}-\frac{x_{2k+1}}{\sqrt{\log x_{2k+1}}}\Big)<x_{2k+1},
$$
which is clearly a contradiction.

{\it Case III.} $q_1,q_2,q_3\not\in Q_{2k}$. This is impossible since $q_1+q_2+q_3<x_{2k+1}.$

{\it Case IV.} $q_1\not\in Q_{2k}$ and $q_2, q_3\in Q_{2k}$. By (\ref{uv}) and (\ref{lastone}), 
$$
x_{2k+1}-q_1\neq q_2+q_3,
$$
which is a contradiction.
\end{proof}

\section*{Acknowledgments}
We thank Yong-Gao Chen and Honghu Liu for their interest in this work.

Y. D. was supported by National Natural Science Foundation of China  (Grant No. 12201544) and China Postdoctoral Science Foundation (Grant No. 2022M710121).

C. E. was supported by a joint FWF-ANR project ArithRand (I 4945-N and ANR-20-CE91-0006).

\end{document}